\documentclass[12pt]{amsart}
\usepackage {enumerate}
\usepackage{setspace}
\usepackage{caption}
\usepackage{mathrsfs}
\usepackage{hyperref}
\usepackage{esint}
 \usepackage{cleveref}
\usepackage{amssymb}
\usepackage{graphicx}
\usepackage{epstopdf}
\usepackage{amsthm}
\usepackage{appendix}
\usepackage{color}
\usepackage{lipsum}
\usepackage{seqsplit}
\renewcommand{\seqinsert}{\ifmmode\allowbreak\else\-\fi}

\newtheorem{theorem}{Theorem}[section]
\newtheorem{lemma}[theorem]{Lemma}

\newtheorem{proposition}[theorem]{Proposition}

\numberwithin{equation}{section}

\theoremstyle {definition}

\makeatletter
\@namedef{subjclassname@2020}{%
  \textup{2020} Mathematics Subject Classification}
\makeatother

\newcommand{\R}{\mathbb{R}}
\newcommand{\N}{\mathbb{N}}
\newcommand{\C}{\mathbb{C}}
\newcommand{\osc}{\operatorname{osc}}
\newcommand{\tr}{\operatorname{tr}}

\newcommand{\dd}{\,\mathrm{d}}

\begin{document}
\title[]{Failure of Interior Hölder and Gradient Estimates for the Subcritical Special Lagrangian Equation}
 
\begin{abstract}
    For every $n\ge3$, we show, by a family of explicit
	entire real-analytic solutions, that, for every subcritical phase and every $0<\alpha\le1$, oscillation-based interior $C^\alpha$ estimates fail for the special Lagrangian equation in $\R^n$. In fact, this family has no common interior modulus of continuity.
\end{abstract}

\author{Zhenyu Fan}
\address{Zhenyu Fan, School of Mathematical Sciences, Peking University, Beijing 100871, China}
\email{fanzhenyu@stu.pku.edu.cn}

\author{Caiyan Li}
\address{Caiyan Li, School of Mathematical Sciences, Xiamen University, Xiamen 361005, China}
\email{caiyanli@xmu.edu.cn}

\author{Zhehui Wang}
\address{Zhehui Wang, School of Sciences, Great Bay University, Dongguan 523000, China}
\email{wangzhehui@gbu.edu.cn}

\maketitle
\section{Introduction}\label{Intro}

The special Lagrangian equation arises from the calibrated geometry of Harvey--Lawson \cite{HL}.  For a $C^2$ potential $u$, the gradient (Lagrangian) graph $\Gamma_u=\{(x,Du(x)) \}\subset\C^n$ is called special Lagrangian with constant phase $\Theta$ precisely when
\begin{equation}\label{eq:SL}
F(D^2u):=\sum_{j=1}^{n}\arctan\lambda_j(D^2u)=\Theta,
\end{equation}
where $\lambda_j(D^2u)$ are the eigenvalues of the Hessian and each arctangent takes values in $(-\pi/2,\pi/2)$. Moreover, $\Gamma_u$ is special if and only if it is a volume minimizing submanifold in $(\R^n \times \R^n, \dd x^2+ \dd y^2)$.  Following Yuan \cite{Yuan06}, the phase is called \textit{critical} when $|\Theta|=(n-2)\pi/2$, \textit{supercritical} when $|\Theta|>(n-2)\pi/2$, and \textit{subcritical} when $|\Theta|<(n-2)\pi/2$.  The level set of \eqref{eq:SL} is convex if and only if the phase lies in the critical and supercritical ranges.

For $0<\alpha\le1$, define
$$
[u]_{C^\alpha(B_r)}
:=\begin{cases}
\displaystyle\sup_{\substack{x,y\in B_r\\x\ne y}}\frac{|u(x)-u(y)|}{|x-y|^\alpha},&0<\alpha<1,\\[6pt]
\|Du\|_{L^\infty(B_r)},&\alpha=1.
\end{cases}
$$

Our main result is as follows.

\begin{theorem}\label{thm:main}
 For every $n\ge3$ and every fixed phase $|\Theta|<(n-2)\pi/2$, there exists a sequence of entire real-analytic solutions $u_k$ to \eqref{eq:SL} in $\R^n$ such that, for every $\alpha\in(0,1]$ and every fixed $r\in(0,1)$,
\[
\sup_{k\ge1}\osc_{B_1}u_k<\infty,
 \qquad\text{but}\qquad
[u_k]_{C^\alpha(B_r)}\longrightarrow\infty \quad\text{as }k\to\infty.
\]
Moreover, for every fixed $r\in(0,1)$, the family $\{u_k\}$ has no common modulus of continuity on $B_r$.
\end{theorem}

	 Warren--Yuan \cite{WY09} obtained explicit gradient and Hessian 
	estimates for every constant phase in dimension two.  Their later work
	\cite[Theorem~1.3]{WY10} proved in every dimension that a smooth solution
	of \eqref{eq:SL} on $B_{3R}$ with constant critical or supercritical phase
	satisfies
	\[
	\max_{B_R}|Du|
	\le C(n)\left(1+\frac{\osc_{B_{3R}}u}{R}\right).
	\]
	Later, Yuan removed the additive dimensional constant in his lecture notes \cite{Yua15}, obtaining the scale-invariant refinement
	\[
	\max_{B_R}|Du|
	\le C(n)\frac{\osc_{B_{2R}}u}{R}.
	\]
	The Warren--Yuan paper also established Hessian estimates in dimension three, and
	Wang--Yuan \cite{WY14} obtained the corresponding Hessian estimates in
	arbitrary dimension.  Together with approximation and the Evans--Krylov
	theory, these estimates yield local analyticity of continuous viscosity
	solutions in the constant critical and supercritical regimes.
	Bhattacharya--Mooney--Shankar \cite{BMS} subsequently extended the
	interior gradient estimate to the Lagrangian mean curvature equation with
	$C^2$ variable critical or supercritical phase; the bound depends on the
	oscillation of $u$ and the $C^2$ norm of the phase. Since an interior gradient estimate immediately implies interior $C^\alpha$ estimates for every $0<\alpha<1$, the results above control the full range $0<\alpha\le1$ in terms of the oscillation of $u$ in the critical and supercritical regimes.

    
	The subcritical picture is different. In dimension three,
	Nadirashvili--Vl\u{a}du\c{t} \cite{NV10} constructed singular
	$C^{1,1/3}$ viscosity solutions at subcritical phases.  For every
	$m\ge2$ and each fixed $\Theta\in(-\pi/2,\pi/2)$, Wang--Yuan
	\cite{WY13} constructed viscosity solutions that belong to
	$C^{1,1/(2m-1)}(B_1)\cap C^\infty(B_1\setminus\{0\})$ but do not
	belong to $C^{1,\delta}(B_1)$ for any $\delta>1/(2m-1)$.  They also produced
	smooth solutions $u^\varepsilon$ satisfying
	\[
	\|Du^\varepsilon\|_{L^\infty(B_1)}\le C,
	\qquad
	|D^2u^\varepsilon(0)|\longrightarrow\infty
	\quad\text{as }\varepsilon\to0.
	\]
		Adding quadratic terms in the remaining variables extends these examples
	to every subcritical phase in higher dimensions. These examples rule out general $C^2$ regularity and Hessian
	estimates.  Since all the above singular solutions are at least $C^1$, there was a conjecture \cite{NV10} that all viscosity solutions to \eqref{eq:SL} are $C^1$.  More recently,
	Mooney--Savin \cite{MS24} disproved this conjecture by constructing a Lipschitz
	viscosity solution at a subcritical phase in dimension three.  Their semi-convex Lipschitz solution is analytic away from an
	embedded surface but not $C^1$; its gradient is discontinuous across the
	surface and its gradient graph is not minimal. Mooney and Shankar \cite{MS25} later generalized Mooney-Savin's singular solutions to any subcritical phase $\Theta\in (-(n-2)\pi/2, (n-2)\pi/2)$ in all dimensions. A natural open problem is whether interior gradient estimates hold for the subcritical phases. If this holds, the examples by Mooney-Savin and Mooney-Shankar would imply that Lipschitz is the optimal regularity we can obtain for viscosity solutions. Mooney \cite{Mooney24} investigated a homogeneous approach to this problem and obtain an obstruction, but that approach does not itself produce a counterexample.  Our theorem above provides a negative answer to this open problem. More strongly, the resulting family has no common interior modulus of continuity; in particular, oscillation-based interior $C^\alpha$ estimates fail for every $0<\alpha<1$, even among entire real-analytic solutions.

	Setting $M=\sup_{k>0}\osc_{B_1}u_k$ in 
	\Cref{thm:main} immediately rules out
	any finite estimate of the form
	\[
	|Du(0)|\le C(n,\Theta,M)
	\qquad\text{whenever }\osc_{B_1}u\le M.
	\]
	Together with the constant-phase estimate of \cite{WY10}, this proves that
	$|\Theta|=(n-2)\pi/2$ is the exact phase threshold for
	oscillation-based interior gradient estimates without additional
	convexity or Hessian constraints. 
	
	The key idea of the construction is to rotate the eigendirections of a fixed 
	quadratic form increasingly rapidly, without increasing the size of its
	coefficients.  Precisely, in the three-dimensional zero-phase case, we start with the simplest two variable saddle quadratic $(x^2-y^2)/2$. Its gradient graph is a Lagrangian 2-plane in $\C^2$ with Lagrangian angles $\pi/4$ and $-\pi/4$. By applying a simultaneous rotation to this Lagrangian 2-plane with constant speed in the first two complex coordinates and a translation along the third complex direction, we obtain a helicoid-like special Lagrangian submanifold in $\C^3$, which is a special case of Joyce's helicoidal examples \cite{JoyceQuadrics,JoyceSymmetry}. As we increase the rotation speed, at the potential level, the rotated quadratic remains uniformly bounded on a fixed ball. Geometrically, however, the height of the gradient graph in the third complex direction will grow arbitrarily large. Finally, by a fixed translation in the first two variables, this produces a family with unbounded gradients at the origin. Along the $x_3$-axis, the potentials also change by a fixed amount over distances of order $k^{-1}$, forcing the H\"older seminorms to diverge and ruling out a common modulus of continuity. For a general subcritical phase $\Theta\in (-\pi/2, \pi/2)$, it suffices to start the rotation from a Lagrangian 2-plane with Lagrangian angles $\Theta/2 + \pi/4$ and $\Theta/2 - \pi/4$. In higher dimensions, adding quadratic terms of the remaining variables will give us desired examples.

\section{Proof of Theorem \ref{thm:main}}
In this section, we first introduce a family of entire real-analytic solutions in dimension three, whose oscillations remain uniformly bounded while their gradients become unbounded. Then, starting from this three-dimensional examples, we prove Theorem \ref{thm:main}.

    For $t\in\R$, set
	$$
	R_t=
	\begin{pmatrix}
		\cos t&-\sin t\\
		\sin t&\cos t
	\end{pmatrix},
	\qquad
	J=\begin{pmatrix}1&0\\0&-1\end{pmatrix}.
	$$
	For $k\in\N_+$, define
	\begin{equation}\label{eq:Qdef}
		Q_k(z):=R_{k z}^{T}JR_{k z}
		=
		\begin{pmatrix}
			\cos(2k z)&-\sin(2k z)\\
			-\sin(2k z)&-\cos(2k z)
		\end{pmatrix}.
	\end{equation}
    The following identities are immediate from this representation.
	
	\begin{lemma}\label{lem:Qidentities}
		Let $I_n$ be the $n\times n$ identity matrix. For every $z\in\R$, 
		\begin{align}
			&Q_k^2=I_2,\qquad \tr Q_k=0,\label{eq:Qbasic}\\
			&Q_k''=-4k^2Q_k, \qquad (Q_k')^2=4k^2I_2,\label{eq:Qderiv}\\
			&Q_k Q_k'+Q_k'Q_k=0.\label{eq:Qanti}
		\end{align}
	\end{lemma}
    Fix a phase $-\frac{\pi}{2}<\theta<\frac{\pi}{2}$. Let
	\begin{equation}\label{eq:Pdef}
		P_{k,\theta}(z)
		=(\tan\theta)I_2+(\sec\theta)Q_k(z).
	\end{equation}
	  Its eigenvalues are $\tan\theta+\sec\theta = \tan(\theta/2+\pi/4)$ and $\tan\theta-\sec\theta= \tan(\theta/2-\pi/4)$, both are independent of $z$. For $q=(x,y)^T$, define
	\begin{equation}\label{eq:U3}
		U_{k,\theta}(x,y,z)
		=\frac12q^TP_{k,\theta}(z)q.
	\end{equation}
	Equivalently,
	\begin{equation}\label{eq:U3expanded}
		U_{k,\theta}(x,y,z)
		=\frac{\tan\theta}{2}(x^2+y^2)+\frac{\sec\theta}{2}\left[
		\cos(2k z)(x^2-y^2)
		-2\sin(2k z)xy
		\right].
	\end{equation}
\begin{lemma}\label{prop:3dphase}
		For every $k\in\N_+$ and every $\theta\in(-\pi/2,\pi/2)$,
		\begin{equation}\label{eq:3dphase}
			\sum_{j=1}^{3}\arctan\lambda_j
			\bigl(D^2U_{k,\theta}\bigr)=\theta
			\qquad\text{in }\R^3.
		\end{equation}
	\end{lemma}

\begin{proof}
    Write $P=P_{k,\theta}(z)$, and set $w=P'q$ and $c=\frac12q^TP''q$. Then the Hessian has the block form
		\begin{equation}\label{eq:Hblock}
			H:=D^2U_{k,\theta}
			=\begin{pmatrix}P&w\\w^T&c\end{pmatrix}.
		\end{equation}
		For a symmetric \(3\times3\) matrix, let \(\sigma_l=\sigma_l(H)\) denote the elementary symmetric functions of its eigenvalues.  A direct computation gives
		\begin{align}
			\sigma_1&=\tr P+c,\label{eq:sigma1}\\
			\sigma_2&=\det P+c\,\tr P-|w|^2,\label{eq:sigma2}\\
			\sigma_3&=c\det P-(\det P)w^TP^{-1}w.\label{eq:sigma3}
		\end{align}
        By \eqref{eq:Qbasic}, we know $P^{-1}=-(\tan\theta)I_2+(\sec\theta)Q_k.$ Hence, \eqref{eq:sigma1}--\eqref{eq:sigma3} yields
		\begin{equation}\label{eq:keyreduce}
			(\sigma_1-\sigma_3)-(1-\sigma_2)\tan\theta
			=2c(\sec\theta)^2-\sec\theta\,w^TQ_k w.
		\end{equation}
		By Lemma~\ref{lem:Qidentities}, we have
		$$
		w=(\sec\theta)Q_k'q, \qquad
		c=\frac{\sec\theta}{2}q^TQ_k''q
		=-2(\sec\theta) k^2q^TQ_k q,
		$$
		and 
		$$
		w^TQ_k w
		=-4(\sec\theta)^2k^2q^TQ_k q.
		$$
		The right-hand side of \eqref{eq:keyreduce} therefore vanishes, so
		\begin{equation}\label{eq:polyphase}
			\sigma_1-\sigma_3=(1-\sigma_2)\tan\theta.
		\end{equation}
		
		It remains to identify the correct branch of the arctangent equation.  From
		\eqref{eq:sigma2},
		\begin{align*}
			1-\sigma_2
			&=2-2c\tan\theta+|w|^2\\
			&=2+4(\sec\theta)k^2
			\left((\sec\theta)|q|^2
			+(\tan\theta)q^TQ_k q\right)\\
            &\ge 2+4(\sec\theta)k^2(\sec\theta-|\tan\theta|)|q|^2,
		\end{align*}
        and hence 
        \begin{equation}\label{eq:branchpositive}
        1-\sigma_2>0.
        \end{equation}
        For $j=1, 2, 3$, set $\alpha_j=\arctan \lambda_j(H)$, and  $$S=\alpha_1+\alpha_2+\alpha_3, \qquad D=\sqrt{1+\lambda_1^2}\sqrt{1+\lambda_2^2}\sqrt{1+\lambda_3^2}.$$
        By a direct calculation, 
        \begin{equation}\label{coss}
            \cos S=\frac{1-\sigma_2}{D},\qquad \sin S=\frac{\sigma_1-\sigma_3}{D}.
        \end{equation}
        Hence, 
        \begin{equation}\label{tans}\tan S=\frac{\sigma_1-\sigma_3}{1-\sigma_2}=\tan\theta,\end{equation} and it yields $S=\theta+m\pi$ for some $m\in\N$. Moreover, by \eqref{eq:branchpositive} and \eqref{coss} we know $\cos S>0$. Combining this with $S\in(-\frac{3\pi}{2}, \frac{3\pi}{2})$, we have $S\in(-\frac{\pi}{2}, \frac{\pi}{2})$. Since $\theta\in(-\frac{\pi}{2}, \frac{\pi}{2})$, by \eqref{tans} we have $S=\theta$. This completes the proof.
\end{proof}

Define
	\begin{equation}\label{eq:Utranslated}
		\widetilde U_{k,\theta}(x,y,z)
		=U_{k,\theta}\left(x+\frac1{\sqrt2},
		y+\frac1{\sqrt2},z\right).
	\end{equation}
	Since translation preserves the Hessian, $\widetilde U_{k,\theta}$ is also an entire solution in $\R^3$.
	
	\begin{proposition}\label{prop:osc-grad}
		For every $k\in\N_+$ and every $\theta\in(-\pi/2,\pi/2)$, we have
		\begin{align}
			\osc_{B_1}\widetilde U_{k,\theta}
			&\le 4(\sec\theta+|\tan\theta|),\label{eq:3dosc}\\
			\left|D\widetilde U_{k,\theta}(0)\right|
			&\ge k\sec\theta.\label{eq:3dgrad}
		\end{align}
	\end{proposition}
	
	\begin{proof}
    Recall the eigenvalues of $P_{k,\theta}(z)$ are $\tan\theta+\sec\theta$ and $\tan\theta-\sec\theta$. For $(x,y,z)\in B_1$, one has
		$$
		\left|\widetilde U_{k,\theta}(x,y,z)\right|
		\le 2(\sec\theta+|\tan\theta|),
		$$
		which proves \eqref{eq:3dosc}.
		
		At \(z=0\),
		$$
		Q_k'(0)
		=-2k
		\begin{pmatrix}0&1\\1&0\end{pmatrix}.
		$$
		and hence 
		\begin{align*}
			\partial_z\widetilde U_{k,\theta}(0)
			&=\frac12q_0^TP_{k,\theta}'(0)q_0\\
			&=\frac{\sec\theta}{2}q_0^TQ_k'(0)q_0
			=-k\sec\theta,
		\end{align*}
        where $q_0=(\frac{1}{\sqrt{2}}, \frac{1}{\sqrt{2}})^T$.
		Thus \eqref{eq:3dgrad} follows.
	\end{proof}

	  Let $n\ge3$ and $|\Theta|<\frac{(n-2)\pi}{2}.$ Set
	\begin{equation}\label{eq:thetahigh}
		\theta=\frac{\Theta}{n-2},
	\end{equation}
	then $|\theta|<\pi/2$.  Define
	\begin{equation}\label{eq:un}
		u_k(x_1,\ldots,x_n)
		=\widetilde U_{k,\theta}(x_1,x_2,x_3)
		+\frac{\tan\theta}{2}\sum_{j=4}^{n}x_j^2,
	\end{equation}
	where the final sum is empty when \(n=3\).
	
	\begin{proof}[Proof of \Cref{thm:main}]
		The Hessian of \eqref{eq:un} is block diagonal:
        \begin{equation}
			D^2u_k=
			\begin{pmatrix}D^2\widetilde U_{k,\theta}&0\\0&(\tan\theta)I_{n-3}\end{pmatrix}.
		\end{equation}
		By Lemma \ref{prop:3dphase}, the three-dimensional block has phase $\theta$.  Each of the remaining $n-3$ eigenvalues equals $\tan\theta$ and contributes
		$$
		\arctan(\tan\theta)=\theta,
		$$
		since $\theta\in(-\pi/2,\pi/2)$.  Consequently,
		$$
		\sum_{j=1}^n\arctan\lambda_j(D^2u_k)
		=\theta+(n-3)\theta=\Theta \qquad \text{in }\R^n.
		$$
        
		At the origin, the added quadratic term has zero gradient.  Hence Proposition~\ref{prop:osc-grad} gives
		$$
		|Du_k(0)|
		\ge |\partial_{x_3}u_k(0)|
		=k\sec\theta,
		$$
		which tends to infinity as $k\to\infty$. Hence, for every fixed $r\in(0,1)$,
$$
[u_k]_{C^1(B_r)}
=\|Du_k\|_{L^\infty(B_r)}
\ge |Du_k(0)|
\ge k\sec\theta
\longrightarrow\infty.
$$

 We next show that the family has no common interior modulus of continuity. Fix $r\in(0,1)$. For $k>\pi/(4r)$, the points $0$ and $\frac{\pi}{4k}e_3$ both lie in $B_r$. By the definition of $\widetilde U_{k,\theta}$,
$$
u_k(te_3)
=
\frac{\tan\theta}{2}
-\frac{\sec\theta}{2}\sin(2kt).
$$
In particular,
$$
\left|
u_k\left(\frac{\pi}{4k}e_3\right)-u_k(0)
\right|
=
\frac{\sec\theta}{2}.
$$
Suppose, to the contrary, that $\{u_k\}$ admits a common modulus of continuity $\rho$ on $B_r$. Thus $\rho(s)\to0$ as $s\downarrow0$ and
$$
|u_k(x)-u_k(y)|
\le
\rho(|x-y|)
$$
for all $k$ and all $x,y\in B_r$. Evaluating this estimate at $x=\frac{\pi}{4k}e_3$ and $y=0$, we obtain
$$
\frac{\sec\theta}{2}
\le
\rho\left(\frac{\pi}{4k}\right)
\longrightarrow0
\qquad\text{as }k\to\infty,
$$
which is impossible. Hence $\{u_k\}$ admits no common modulus of continuity on $B_r$.

The same pair of points gives, for every $0<\alpha<1$,
$$
[u_k]_{C^\alpha(B_r)}
\ge
\frac{
\left|
u_k\left(\frac{\pi}{4k}e_3\right)-u_k(0)
\right|
}{
\left|\frac{\pi}{4k}e_3\right|^\alpha
}
=
\frac{\sec\theta}{2}
\left(\frac{4k}{\pi}\right)^\alpha
\longrightarrow\infty.
$$
Therefore, $\{u_k\}$ is not uniformly bounded in $C^\alpha(B_r)$ for any $0<\alpha<1$. 

It remains to verify the uniform oscillation bound. In fact, 
		$$
		\osc_{B_1}\left(\frac{\tan\theta}{2}
		\sum_{j=4}^{n}x_j^2\right)
		\le\frac{|\tan\theta|}{2}.
		$$
		Combining this with \eqref{eq:3dosc} gives
		$$
		\osc_{B_1}u_k
		\le4(\sec\theta+|\tan\theta|)
		+\frac{|\tan\theta|}{2},
		$$
		which is independent of $k$. This completes the proof.
	\end{proof}

\vspace{3em}
{\bf Disclosure on AI assistance.} The authors used AI-assisted tools, principally ChatGPT. The authors wrote and verified all theorem statements, proofs, and they take full responsibility for the contents of the paper.

\end{document}